\documentclass[11pt]{amsart}
\usepackage{amscd,amssymb,verbatim} 

\theoremstyle{plain}
\newtheorem{thm}{Theorem}
\newtheorem{cor}[thm]{Corollary}
\newtheorem{lem}[thm]{Lemma}
\newtheorem{prop}[thm]{Proposition}
\newtheorem{Def}[thm]{Definition}

\newcommand{\ep}{\varepsilon}
\newcommand{\al}{\alpha}

\newcommand{\Sig}{\Sigma}
\newcommand{\vp}{\varphi}

\newcommand{\bs}{\backslash}

\newcommand{\ol}{\overline}
\newcommand{\ti}{\widetilde}

\newcommand{\Lip}{\operatorname{Lip}}
\newcommand{\lip}{\operatorname{lip}}
\newcommand{\rg}{\operatorname{range}}
\newcommand{\loc}{\operatorname{loc}}
\newcommand{\spn}{\operatorname{span}}

\newcommand{\N}{{\mathbb N}}
\newcommand{\R}{{\mathbb R}}

\newcommand{\C}{{\mathbb C}}

\newcommand{\cZ}{{\mathcal Z}}
\newcommand{\cA}{{\mathcal A}}
\newcommand{\cF}{{\mathcal F}}
\newcommand{\cX}{{\mathcal X}}

\begin{document}

\title{Order isomorphisms on function spaces}

\begin{abstract}
The classical theorems of Banach and Stone \cite{B, St}, Gelfand and Kolmogorov \cite{GK} and Kaplansky \cite{K} show that a compact Hausdorff space $X$ is uniquely determined by the linear isometric structure, the algebraic structure, and the lattice structure, respectively, of the space $C(X)$.  In this paper, it is shown that for rather general subspaces $A(X)$ and $A(Y)$ of $C(X)$ and $C(Y)$ respectively, any linear bijection $T: A(X) \to A(Y)$ such that $f \geq 0$ if and only if $Tf \geq 0$ gives rise to a homeomorphism $h: X \to Y$ with which $T$ can be represented as a weighted composition operator.  The three classical results mentioned above can be derived as corollaries. Generalizations to noncompact spaces and other function spaces such as spaces of uniformly continuous functions, Lipschitz functions and differentiable functions are presented.
\end{abstract}

\author{Denny H.\ Leung}
\address{Department of Mathematics, National University of Singapore, Singapore 119076}
\email{matlhh@nus.edu.sg}
\author{Lei Li}
\address{School of Mathematical Sciences and LPMC, Nankai University, Tianjin, 300071, China}
\email{leilee@nankai.edu.cn}
\thanks{Research of the first author was partially supported by AcRF project no.\ R-146-000-157-112}

\maketitle

\section{Introduction}

A well established area of research seeks to determine the relationship between the structure of a set and the properties of certain function spaces defined on it.  A closely related question is to determine the general form of an operator mapping between various function spaces.  In the case of spaces of continuous functions, a vast literature exists. A good summary of the classical results is the text of Gillman and Jerison \cite{GJ}.  From the classical period, three results in this area stand out; namely the theorems of Banach and Stone, Gelfand-Kolmogorov and Kaplansky (see Corollary \ref{cor5.2} below).  When $X$ and $Y$ are compact Hausdorff spaces, these results determine the precise forms of the norm isometries, algebra isomorphisms and vector lattice isomorphisms between $C(X)$ and $C(Y)$ respectively.  In particular, the existence of any one of these three types of operators lead to homeomorphism between $X$ and $Y$.  More recent results of Banach-Stone type are found in \cite{A1, A2, GaJ}, for example.

Kadison \cite{Ka} showed that a linear order isomorphism $T$ between
two $C^*$-algebras which maps the identity to the identity is a
$C^*$-isomorphism, i.e., $T$ satisfies $T(x^2) = (Tx)^2$ for all
$x$. In the commutative case, it follows that for any compact
Hausdorff spaces $X$ and $Y$, any linear order isomorphism from
$C(X)$ onto $C(Y)$ that maps the constant function $1$ to the
constant function $1$ is an algebra isomorphism. The theorem of
Gelfand-Kolmogorov then implies that $T$ is a composition operator.
Order isomorphisms, even nonlinear ones, have been studied by
various authors, see, e.g.,  \cite{CS1, CS2, CS3,  GaJ, GaJ2,
JV2008, JV2009}.  These results concern order isomorphisms between
specific function spaces.  Moreover, the spaces considered are all
either lattices or algebras of functions.  One of the aims of the
present paper is to provide a unified treatment of linear order
isomorphisms within a general framework.  In particular, our results
apply to all unital function lattices that separate points from
closed sets and many function algebras.

For quite general subspaces $A(X)$ and $A(Y)$ of $C(X)$ and $C(Y)$
respectively, where $X$ and $Y$ are compact Hausdorff spaces,
Theorem \ref{thm0} in \S\ref{sec1} determines the precise form of a
linear order isomorphism $T: A(X) \to A(Y)$ and shows that the
existence of such a map leads to homeomorphism of $X$ and $Y$.  The
classical results cited above can all be subsumed under this
theorem. We also provide an example of a space to which the theorem
applies which is neither a lattice nor an algebra.

In \S \ref{sec2}, we apply a Stone-\v{C}ech like compactification
procedure  to extend Theorem \ref{thm0} to noncompact spaces. As a
result, existence of a linear order isomorphism $T:A(X)\to A(Y)$
gives rise to a homeomorphism between some compactifications of $X$
and $Y$ respectively.  In \S \ref{sec3}, we show that under certain
circumstances, the homeomorphism obtained restricts to a
homeomorphism between $X$ and $Y$.

Because much of the paper is concerned with maps preserving order, we consider only {\em real} vector spaces.

The authors wish to thank the referee whose astute comments on the
first submission of this paper prompted much further reflection and
great improvements in both the content and exposition of the paper.

\section{Order isomorphisms of spaces of continuous functions on compact Hausdorff spaces}\label{sec1}

Let $X$ be a topological space and let $A(X)$ be a vector subspace
of $C(X)$. $A(X)$ is said to {\em separate points from closed sets}
if given $x\in X$ and  a closed set $F$ in $X$ not containing $x$,
there exists $f \in A(X)$ such that $f(x) = 1$ and $f(F) \subseteq
\{0\}$. If, in addition, $f$ can be chosen to have values in
$[0,1]$, then we say that $A(X)$ {\em precisely separates points
from closed sets}. It is clear that any sublattice of $C(X)$ that
separates points from closed sets does so precisely. Let $Y$ be a
topological space and let $A(Y)$ be a vector subspace of $C(Y)$, a
linear bijection $T:A(X) \to A(Y)$ is an {\em order isomorphism} if
$f \geq 0$ if and only if $Tf \geq 0$. The aim of this section is to
prove the following theorem.

\begin{thm}\label{thm0}
Let $X$, $Y$ be compact Hausdorff spaces and let $A(X)$ and $A(Y)$ be subspaces of $C(X)$ and $C(Y)$ respectively that contain the constant functions and precisely separate points from closed sets. If $T: A(X)\to A(Y)$ is a  linear order isomorphism, then there is a homeomorphism $h:X \to Y$ such that $Tf = T1_X\cdot f\circ h^{-1}$ for all $f \in A(X)$.
\end{thm}

The proof is divided into a number of steps listed below, from Proposition \ref{prop1} to Proposition \ref{prop4}. If $f \in A(X)$ or $A(Y)$, let $Z(f) = \{f = 0\}$.

\begin{prop}\label{prop1}
For any $x_0 \in X$, let
\[ \cZ_{x_0} = \{Z(Tf): f \in A(X), f \geq 0, f(x_0) = 0\}.\]
Then $\cZ_{x_0}$ has the finite intersection property.
\end{prop}

\begin{proof}
Suppose that $f_i \in A(X)$, $f_i \geq 0$, and $f_i(x_0) = 0$, $1
\leq i \leq n$. Let $f = \sum^n_{i=1}f_i$.
Then $f \in A(X)$, $f \geq 0$, and $f(x_0) = 0$.
In particular, $Tf \geq 0$.
If $Z(Tf) = \emptyset$, that is, $(Tf)(y)>0$ for all $y\in Y$, there exists $\ep > 0$ such that $Tf - \ep T1_X\geq 0$.
Thus $f - \ep 1_X \geq 0$, which is manifestly untrue. Let $y_0 \in
Y$ be such that $Tf(y_0) = 0$. Then $\sum^n_{i=1}Tf_i(y_0) = 0$.
Since $Tf_i \geq 0$,  $Tf_i(y_0) = 0$, $1 \leq i \leq n$.  Thus $y_0
\in \cap^n_{i=1}Z(Tf_i)$.
\end{proof}

Define $\cZ_{y_0}$ similarly for $y_0 \in Y$, using the operator
$T^{-1}$ in place of $T$. By Proposition \ref{prop1},
$\cap\cZ_{x_0}$ and $\cap\cZ_{y_0}$ are nonempty for all $x_0\in X,
y_0 \in Y$.

\begin{prop}\label{prop2}
Let $x_0 \in X$ and $y_0 \in Y$. Then $y_0 \in \cap\cZ_{x_0}$ if and
only if $x_0 \in \cap\cZ_{y_0}$.
\end{prop}

\begin{proof}
Suppose that $y_0 \in \cap\cZ_{x_0}$ but $x_0 \notin \cap\cZ_{y_0}$.
Choose $x_1 \in \cap\cZ_{y_0}$.  Then $x_1 \neq x_0$. There exists
$f \in A(X)$, $\rg f \subseteq [0,1]$, such that $f(x_1) = 1$ and
$f(x_0) = 0$. Since $y_0 \in \cap\cZ_{x_0}$, $Tf(y_0) = 0$. As $T$
is an order isomorphism, $Tf \geq 0$. Then $x_1 \in \cap\cZ_{y_0}$
implies that $f(x_1) = T^{-1}(Tf)(x_1) = 0$, contrary to the choice
of $f$. The ``if" part of the proposition follows by symmetry.
\end{proof}

\begin{prop}\label{prop3}
$\cap\cZ_{x_0}$ contains exactly one point.
\end{prop}

\begin{proof}
It has already been observed that $\cap \cZ_{x_0}$ is nonempty.
Suppose that there are distinct points $y_1$ and $y_2$ in
$\cap\cZ_{x_0}$. Choose $g\in A(Y)$ such that $\rg g\subseteq
[0,1]$, $g(y_1) = 0$ and $g(y_2) = 1$. By Proposition \ref{prop2},
$x_0 \in \cap\cZ_{y_1}$. Hence $T^{-1}g(x_0) = 0$.
Since
$y_2 \in \cap\cZ_{x_0}$,
$T(T^{-1}g)(y_2) = 0$.   Thus $g(y_2)=0$, yielding a
contradiction.
\end{proof}

\begin{prop}\label{prop3.5}
$T1_X(y) > 0$ and $T^{-1}1_Y(x)> 0$ for all $x\in X$ and all $y\in Y$.
\end{prop}

\begin{proof}
Suppose that there exists $y_0\in Y$ such that $T1_X(y_0) = 0$. For any $f\in A(X)$, there exists $0 \leq c\in \R$ such that $-c 1_X \leq f \leq c1_X$. Then $-cT1_X(y_0)\leq Tf(y_0) \leq c T1_X(y_0)$ for all $f\in A(X)$.  Hence $Tf(y_0) = 0$ for all $f\in A(X)$. This is a contradiction since $T$ maps onto $A(Y)$ and $A(Y)$ contains all constant functions.
\end{proof}

Define $h: X \to Y$ by $h(x_0) = y_0$, where $\{y_0\} =
\cap\cZ_{x_0}$.

\begin{prop}\label{prop4}
$h$ is a homeomorphism from $X$ onto $Y$ so that
$Tf = T1_X\cdot f\circ h^{-1}$
for all $f \in A(X)$ and all $y \in Y$.
\end{prop}

\begin{proof}
The injectivity of $h$ follows from Proposition \ref{prop2}. If $y \in Y$, let $\{x\} = \cap\cZ_{y}$. By Proposition
\ref{prop2}, $y \in \cap\cZ_{x}$.  Thus $h(x) = y$. This shows that
$h$ is surjective.

Suppose that $x_0\in X$ and $y_0 = h(x_0)$.  Let $f \in A(X)$ and
let $m = \min\{f(x): x\in X\}$. Given $\ep > 0$, let $U$ be an open
neighborhood of $x_0$ so that $f(x) > f(x_0) - \ep$ for all $x\in
U$.  There exists $g_1\in A(X)$ such that $\rg g_1 \subseteq [0,1]$,
$g_1(x_0) =1$ and $g_1(x) = 0$ for all $x\notin U$. Since $1_X-g_1 \geq
0$, $(1_X-g_1)(x_0) = 0$ and $y_0\in \cap\cZ_{x_0}$, $T(1_X-g_1)(y_0) =
0$.  Hence $(Tg_1)(y_0) =(T1_X)(y_0)$.
Now \[f - m1_X - (f(x_0)-m -\ep)g_1 \geq 0.\]
Thus
\[(Tf)(y_0) \geq m(T1_X)(y_0) + (f(x_0)-m -\ep)(Tg_1)(y_0),\]
that is, $(Tf)(y_0)\geq (T1_X)(y_0)(f(x_0)-\ep)$.
As $\ep>0$ is arbitrary, $Tf (y_0) \geq (T1_X)(y_0)f(x_0)$.
Applying the argument to $-f$ yields the reverse inequality.
Thus $Tf(y_0) = (T1_X)(y_0)f(x_0)$.

It remains to show that $h$ is a homeomorphism.
Let $x_0 \in X$ and $y_0 = h(x_0)$.
Suppose that $V$ is an open neighborhood of $y_0$ in $Y$.
There exists $g \in A(Y)$ such that $g(y_0) = 1$, $\rg g \subseteq [0,1]$, and $g = 0$ outside $V$.
Since $g \geq 0$, $T^{-1}g \geq 0$.
If $T^{-1}g(x_0) = 0$, then $y_0 \in Z(T(T^{-1}g)) = Z(g)$, contrary to the choice of $g$.
Thus $T^{-1}g(x_0) > 0$.
Therefore, the set $U = \{T^{-1}g > 0\}$ is an open neighborhood of $x_0 \in X$.
Suppose that $x\in U$. By the previous paragraph, \[g(h(x)) = T(T^{-1}g)(h(x)) = (T1_X)(h(x))T^{-1}g(x) > 0. \]
Hence $h(x) \in V$. This proves that $h$ is continuous.
Since $h$ is a continuous bijection between compact Hausdorff spaces, it is a homeomorphism.
\end{proof}

\noindent{\bf Remark.} Applying Theorem \ref{thm0} to the map $T^{-1}$ gives a homeomorphism $k:Y\to X$ such that $T^{-1}g = T^{-1}1_Y\cdot g\circ k^{-1}$.  Because of Proposition \ref{prop2}, $k$ must be $h^{-1}$.

\bigskip

The classical theorems of Banach \cite{B} and Stone \cite{St}, Gelfand and Kolmogorov \cite{GK} and  Kaplansky \cite{K} show that a compact Hausdorff space $X$ is uniquely determined by the linear isometric structure, the algebraic structure, and the lattice structure, respectively, of the space $C(X)$.
These results can be subsumed under Theorem \ref{thm0}.  As usual, $C(X)$ and $C(Y)$ are endowed with their respective supremum norms.

\begin{lem}\label{lem5.1}
Let $X$ and $Y$ be compact Hausdorff spaces.  If $T: C(X) \to C(Y)$ is an onto linear isometry, then $|T1_X| = 1_Y$ and, for any $f\in C(X)$, $f \geq 0$ if and only if $Tf\cdot T1_X \geq 0$.
\end{lem}

\begin{proof}
Since $T$ is an isometry, $\|T1_X\| = 1$.
Suppose that there exists $y_0\in Y$ such that $|T1_X(y_0)|  < 1$.
There exists a neighborhood $V$ of $y_0$ and $a > 0$ such that $|T1_X(y)| < 1-a$ for all $y \in V$.
Choose $g\in C(Y)$ such that $g(y_0) = 1$, $\rg g\subseteq [0,1]$ and $g = 0$ outside $V$.
Then $\|1_X + aT^{-1}g\| = \|T1_X + ag\| \leq 1$. Thus $1+ a T^{-1}g(x) \leq 1$ for all $x\in X$.
Therefore, $T^{-1}g \leq 0$.
As $\|T^{-1}g\| = \|g\| = 1$, there must be some $x_0 \in X$ where $T^{-1}g(x_0) = -1$.
Then
\[ 2 = |(1_X- T^{-1}g)(x_0)| \leq \|1_X - T^{-1}g\| \leq \|1_X\| + \|T^{-1}g\| = 2.\]
Hence $\|T1_X -g\| = \|1_X - T^{-1}g\| = 2$. However, if $y\in V$, then
\[|(T1_X-g)(y)| \leq |T1_X(y)| + \|g\| < 1-a +1 < 2.\]
On the other hand, if $y\notin V$, then $|(T1_X-g)(y)| = |T1_X(y)| \leq \|1_X\| = 1$.
This proves that $\|T1_X-g\|< 2$, contrary to the above.  Therefore, $|T1_X| = 1_Y$.

Given $f \in C(X)$, $f \geq 0$ if and only if $\|f - \|f\|1_X\| \leq \|f\|$. Since $T$ is an isometry, this is equivalent to
 $\|Tf - \|Tf\|T1_X\| \leq \|Tf\|$.
By the above, $|T1_X| = 1_Y$. Thus the final inequality holds if and only if $Tf \cdot T1_X \geq 0$.
\end{proof}

A linear bijection $T:C(X) \to C(Y)$ is
\begin{enumerate}
\item a {\em lattice isomorphism} if $|Tf| = T|f|$ for all $f \in C(X)$;
\item an {\em algebra isomorphism} if $T1_X = 1_Y$ and $T(fg) = Tf\cdot Tg$ for all $f,g\in C(X)$.
\end{enumerate}

\begin{cor}\label{cor5.2}
Let $X$ and $Y$ be compact Hausdorff spaces and let $T: C(X)\to C(Y)$ be a linear bijection.
\begin{enumerate}
\item (Banach-Stone) If $T$ is an isometry, then there is a homeomorphism $h: X\to Y$ and a function $g \in C(Y)$, $|g| = 1_Y$, such that $Tf = g\cdot f\circ h^{-1}$ for all $f\in C(X)$.
\item (Kaplansky) If $T$ is a lattice isomorphism, then  there is a homeomorphism $h: X\to Y$ and a function $g \in C(Y)$, $g(y) > 0$ for all $y\in Y$, such that $Tf = g\cdot f\circ h^{-1}$ for all $f\in C(X)$.
\item (Gelfand and Kolmogorov) If $T$ is an algebra isomorphism, then  there is a homeomorphism $h: X\to Y$  such that $Tf = f\circ h^{-1}$ for all $f\in C(X)$.
\end{enumerate}
\end{cor}

\begin{proof}
Let $g = T1_X$. For case (a), it follows from Lemma \ref{lem5.1} that $|g| = 1_Y$ and that the operator $\ti{T}: C(X)\to C(Y)$ given by $\ti{T}(f) = Tf/g$ is a linear order isomorphism. By Theorem \ref{thm0}, there exists a homeomorphism $h:X\to Y$ such that $\ti{T}f = \ti{T}1_X\cdot f\circ h^{-1}$. It follows easily that $Tf = g\cdot f\circ h^{-1}$. For cases (b) and (c), it is clear that $T$ is a  linear order isomorphism.  Proposition \ref{prop3.5} gives that $g(y) > 0$ for all $y\in Y$. Moreover, $g = 1_Y$ for case (c).
The representation of $T$ follows immediately from Theorem \ref{thm0}.
\end{proof}

\noindent{\bf Remark}. Lemma \ref{lem5.1} may be extended to linear isometries $T: C(X,\C)\to C(Y,\C)$.  As a result, Corollary \ref{cor5.2}(a) for complex isometries may also be derived in a similar manner.

\bigskip

\noindent{\bf Example}. There is a subspace $\cX$ of $C[0,1]$ that contains constants and precisely separates points from closed sets, so that $\cX$ is neither a sublattice nor a subalgebra of $C[0,1]$.

Define $\theta:\R \to \R$ by $\theta(t) = \sin \frac{\pi t}{2}$ and $g:\R\to \R$ by $g(t) = 0, \theta(t), 1$ reespectively if $t \leq 0, 0 < t < 1, t\geq 1$ respectively.    Denote by $t$ the identity function $t\mapsto t$ on $[0,1]$ or $\R$, as the case may be.
If $\cF$ is a set of real-valued functions defined on $[0,1]$ or $\R$, let $g\circ \cF = \{g\circ f: f\in \cF\}$.
Denote by $\cX_1$ and $\Sig_1$ respectively the span of the functions $1$ and $t$ in $C[0,1]$ and $C(\R)$ respectively.
Let $\cX_{n+1} = \spn\{\cX_n \cup g\circ \cX_n\}$ and $\Sig_{n+1} = \spn\{\Sig_n \cup \theta\circ \Sig_n\}$.
Then set $\cX = \cup_n \cX_n$ and $\Sig = \cup_n \Sig_n$.
It is easy to see that $\cX$ is a subspace of $C[0,1]$ and that $\Sig$ is a subspace of $C(\R)$.  Moreover, $\Sig$ consists of real analytic functions on $\R$.
We claim that $\cX$ has the desired properties.

First observe that for any $a < b$ in $[0,1]$, the linear function
$f$ on $[0,1]$ such that $f(a) =0$ and $f(b) = 1$ belongs to
$\cX_1$. Hence $g\circ f \in \cX_2 \subseteq \cX$.  $g\circ f$ has
the property that $g\circ f(t) = 0$ if $t\leq a$, $1$ if $t \geq b$
and $0 < g\circ f(t) < 1$ if $a < t < b$.  Taking differences of two
such functions shows that $\cX$ separates points from closed sets in
$[0,1]$.  By construction, $g\circ f\in \cX$ for any $f\in \cX$.
Thus $\cX$ satisfies conditions (a) and (b) of Definition
\ref{def:adequate} below.  Since $\cX$ consists of bounded
functions, condition (c) is also satisfied.  It follows from Lemma
\ref{lem6.0} that $\cX$ precisely separates points form closed sets
in $[0,1]$.

\begin{lem}\label{lem8}
If $f\in \cX_n$ and $I$ is a nondegenerate interval in $[0,1]$, then there is a nondegenerate interval $J \subseteq I$ and a function $u \in \Sig_n$ such that $f = u$ on $J$.
\end{lem}

\begin{proof}
Induct on $n$.  The case $n=1$ is trivial. Assume that the result
holds for some $n$ and let $f\in \cX_{n+1}$. Then  $f= f_0 +
\sum^m_{i=1}c_ig\circ f_i$, where $f_i \in \cX_n$, $0\leq i \leq m$.
By the inductive assumption, there exist a nondegenerate interval
$I_0 \subseteq I$ and $u_0 \in \Sig_n$ such that $f_0 = u_0$ on
$I_0$. If $f_1(I_0) \cap (0,1) = \emptyset$, then $g\circ f_1(I_0)
\subseteq \{0,1\}$.  By connectedness of $I_0$ and continuity of
$g\circ f_1$, $g\circ f_1$ takes constant value, say $c_1 (= 0$ or
$1$) on $I_0$.  In this case, let $I_1 = I_0$ and $u_1 = c_1$. Then
$g\circ f_1 = c_1 = \theta\circ u_1$ on $I_1$. Otherwise,
$f_1(I_0)\cap (0,1) \neq \emptyset$.  There is a nondegenerate
interval $I_0' \subseteq I_0$ such that $f_1(I_0') \subseteq (0,1)$.
Then $g\circ f_1 = \theta\circ f_1$ on $I_0'$.  By the inductive
assumption, there exist a nondegenerate interval $I_1 \subseteq
I_0'$ and $u_1 \in \Sig_n$ such that $f_1 = u_1$ on $I_1$. Then
$g\circ f_1 = \theta\circ f_1  = \theta\circ u_1$ on $I_1$. Continue
to choose nondegenerate intervals $I_0 \supseteq I_1 \supseteq
\cdots \supseteq I_m$, $u_0, \dots, u_m \in \Sig_n$ such that
$g\circ f_i = \theta\circ u_i$ on $I_i$, $1\leq i \leq m$, and $f_0
= u_0$ on $I_0$. Then $f= u_0 + \sum^m_{i=1}c_i \theta\circ u_i$ on
$I_m$ and the latter function belongs to $\Sig_{n+1}$.
\end{proof}

\begin{lem}\label{lem8.1}
If $f$ is a real analytic function on $\R$ and $f_{|[0,1]}\in \cX$, then $f\in \Sig$.
\end{lem}

\begin{proof}
Let $f_0 = f_{|[0,1]}$.  By Lemma \ref{lem8}, there is a nondegenerate interval $J$ and $u \in \Sig$ such that $f_0 = u$ on $J$.  Since both $f$ and $u$ are real analytic on $\R$, $f = u$ on $\R$.
\end{proof}

We can now verify the remaining properties of $\cX$ stated above,
namely that $\cX$ is neither a sublattice nor a subalgebra of
$C[0,1]$.  If $\cX$ is an algebra, then $t^2 \in \cX$.  By Lemma
\ref{lem8.1}, $t^2$ (as a function on $\R$) belongs to $\Sig$.
We show  that this is impossible by showing that
$\lim_{t\to\infty}f(t)/t^2 = 0$ for any $f\in \Sig$. Indeed, the
statement holds for any $f\in \Sig_1$.  Inductively, any $f\in
\Sig_{n+1}$ can be written as $f = f_0 +\sum^m_{i=1}c_i\theta\circ
f_i$, where $f_i\in \cX_n$, $0 \leq i \leq m$. By induction,
$\lim_{t\to\infty}f_0(t)/t^2 = 0$.  Since $0\leq \theta\circ f_i\leq
1$, the induction is complete.

Finally, suppose that $\cX$ is a lattice.  Then $t \wedge
\frac{1}{2} \in \cX$.  Say it belongs to $\cX_n$. It is easy to
check via induction that
\[\cX_n = \cX_1 + \spn(g\circ \cX_1) + \cdots + \spn(g\circ \cX_{n-1}).\]
Then we can write $f = f_0 + g\circ f_1 + \cdots + g\circ f_{n-1}$, where $f_0 \in \cX_1$, $f_i \in \cX_i$, $1\leq i \leq n-1$.  Say $f_0 = a + bt$.
By Lemma \ref{lem8}, there exist nondegenerate intervals
$J_1 \subseteq [0,1/2]$, $J_2 \subseteq [1/2,1]$, functions $u_i,v_i\in \Sig$, $1\leq i\leq n-1$, such that
\[ f= \begin{cases}
     a+ bt + \sum^{n-1}_{i=1}\theta\circ u_i & \text{on $J_1$}\\
      a+ bt  +  \sum^{n-1}_{i=1}\theta\circ v_i &\text{on $J_2$}.
   \end{cases}\]
Then $t = a+bt + h_1$ on $J_1$ and $1/2 = a+bt+ h_2$ on $J_2$, where $h_1$ and $h_2$ are bounded real analytic functions.  Thus these equations hold on $\R$.
Dividing both equations by $t$ and taking limits as $t\to \infty$ gives $1 = b$ and $0 = b$, which is absurd.

\section{Order isomorphisms on completely regular spaces}\label{sec2}

In this section, we employ the method of compactification to extend
results in the previous section to more general function spaces. Let
$X$ be a (Hausdorff) completely regular space and let $A(X)$ be a
subspace of $C(X)$ that separates points from closed sets. Denote by
$\R_\infty$ the interval $[-\infty,\infty]$ with the order topology.
The map $i : X\to \R^{A(X)}_\infty$, $i(x)(\vp) = \vp(x)$, is a
homeomorphic embedding. Let $\cA X$ be the closure of $i(X)$ in
$\R^{A(X)}_\infty$.  Then $\cA X$ is a compact Hausdorff space.
Identify $X$ with $i(X)$ and regard $X$ as a subspace of $\cA X$.
For each $f\in A(X)$, there is a unique continuous extension
$\hat{f}:\cA X \to \R_\infty$ given by $\hat{f}(x) = x(f)$. When
$A(X) = C(X)$, $\cA X$ coincides with the Stone-\v{C}ech
compactification $\beta X$. In some versions of this type of
compactification, one embeds $\R$ into the one point
compactification $\R \cup\{\infty\}$.  We prefer to use the
compactification $[-\infty,\infty]$ instead since this space is
ordered. Let $g:\R\to \R$ be a function.  We say that $A(X)$ is
$g$-{\em invariant} if $g\circ f \in A(X)$ for all $f\in A(X)$.

\begin{Def}\label{def:adequate}
We will say that a vector subspace $A(X)$ of $C(X)$ is {\em adequate} if
\begin{enumerate}
\item $A(X)$ separates points from closed sets and contains the constant functions;
\item There exists a continuous nondecreasing function $g: \R \to \R$, with $g(t) = 0$ if $t \leq 0$ and $g(t) = 1$ if $t \geq 1$, such that $A(X)$ is $g$-invariant.
\item The positive cone $A(X)_+$ generates $A(X)$, i.e., every $f\in A(X)$ can be written as $f_1 - f_2$, where $f_1$ and $f_2$ are nonnegative functions in $A(X)$.
\end{enumerate}
\end{Def}

Observe that conditions (b) and (c) hold if $A(X)$ is a sublattice
of $C(X)$.  Indeed, condition (c) is obvious for a sublattice.  Take
$g(t)$ to be $0, t$ and $1$ respectively  for $t \leq 0, 0 < t< 1$
and $t \geq 1$ respectively.  Then $A(X)$ is $g$-invariant if $A(X)$
is a sublattice of $C(X)$. Also, condition (c) holds if $A(X)$
consists of bounded functions and contains constants.  Indeed, in
this case, any $f\in A(X)$ can be written as $c1_X - (c1_X - f)$,
where $c \geq 0$ and $c1_X \geq f$.

\begin{lem}\label{lem5.9}
Let $X$ be a completely regular space and let $A(X)$ be an adequate subspace of $C(X)$,
Given $x_0 \in \cA X$, $f\in A(X)$, and a neighborhood $U$ of $x_0(f)$ in $[-\infty,\infty]$, set $V = \{x\in \cA X: x(f) \in U\}$.  Then there exists $h\in A(X)$, with $0 \leq h \leq 1_X$ and $x_0(h) = 0$, such that $W=\{x\in \cA X:x(h) < 1\} \subseteq V$.
\end{lem}

\begin{proof}
Let $g$ be the function given in Definition \ref{def:adequate}.
Define $g(\infty)= 1$ and $g(-\infty) = 0$. Then $x(g\circ f) =
g(x(f))$ for all $x\in \cA X$ and all $f\in A(X)$. First consider
the case where $x_0(f) =a \in \R$. Choose $\ep
> 0$ such that $(a-\ep,a+\ep) \subseteq U$. Set $f_1 =
\ep^{-1}(f-a1_X) \in A(X)$. Then $x_0(f_1) = 0$.  If $x\in \cA X$
and $|x(f_1)| < 1$, then $|x(f) -a| = \ep|x(f_1)| < \ep$. Thus $W' =
\{x\in \cA X: |x(f_1)| < 1\} \subseteq V$. Set $h = 1_X + g \circ
f_1 - g\circ (f_1 + 1_X) \in A(X)$. Using the observation above, we
see that for any $x\in\cA X$, $x(h) = 1 + g(x(f_1)) - g(x(f_1)+1)$.
It is easy to check that $0 \leq h\leq 1_X$, $x_0(h) = 0$, and that
$W \subseteq W' \subseteq V$.

If $x_0(f) = +\infty$ or $-\infty$, the proof is similar.  Assume the former. Choose $0 < m \in \R$ such that $(m,\infty] \subseteq U$ and define $h = 1_X - g\circ (f - (m+1)1_X)$.  We omit the verification that $h$ satisfies the requirements.
\end{proof}

Denote by $C_b(X)$ the subspace of $C(X)$ consisting of the bounded
functions. Let $A_b(X) = A(X) \cap C_b(X)$ and $\hat{A}_b(X) =
\{\hat{f}: f \in A_b(X)\}$. The map $f \mapsto \hat{f}$ is a
bijection from $A_b(X)$ onto the subspace $\hat{A}_b(X)$ of $C(\cA
X)$.  Since $A_b(X)$ contains the constant functions on $X$,
$\hat{A}_b(X)$ contains the constant functions on $\cA X$.

\begin{lem}\label{lem6.0}
Let $X$ be a completely regular space.  If $A(X)$ is an adequate subspace of $C(X)$, then $\hat{A}_b(X)$ precisely separates points from closed sets in $\cA X$.
\end{lem}

\begin{proof}
Let $x_0 \in \cA X$ and $F$ be a closed subset of $\cA X$ not
containing $x_0$. Then $V = \cA X \bs F$ is an open neighborhood of
$x_0$.  Choose $f_1,\dots, f_n \in A(X)$ and open neighborhoods
$U_i$ of $x_0(f_i)$ in $[-\infty,\infty]$ such that $\cap^n_{i=1}V_i
\subseteq V$, where $V_i = \{x\in \cA X: x(f_i) \in U_i\}$. By Lemma
\ref{lem5.9}, there exist $h_i$, $1\leq i\leq n$, such that $0 \leq
h_i \leq 1_X$, $x_0(h_i) = 0$, and $W_i = \{x\in \cA X: x(h_i) < 1\}
\subseteq V_i$. Set $h = 1_X - g\circ \sum^n_{i=1}h_i \in A(X)$.
Then $0 \leq h \leq 1_X$ and $\hat{h}(x_0)  = x_0(h) = 1$. If $x\in
\cA X$, $x\notin V$, then there exists $j$ such that $x \notin V_j$,
and hence $x\notin W_j$. Thus $x(h_j) \geq 1$ and so
$x(\sum^n_{i=1}h_i) \geq 1$. Therefore, $\hat{h}(x) = x(h) = 0$.
\end{proof}

\begin{thm}\label{thm6}
Let $X$ and $Y$ be completely regular spaces. Suppose that $A(X)$ and $A(Y)$ are adequate subspaces of $C(X)$ and $C(Y)$ respectively. If $T: A(X) \to A(Y)$ is a  linear order isomorphism such that $T(A_b(X)) = A_b(Y)$, then there exists a homeomorphism $h :\cA X \to \cA Y$ such that for all $x\in X, y\in Y$, $f\in A(X)$ and $g\in A(Y)$, $\hat{f}(h^{-1}(y)), \hat{g}(h(x)) \in \R$, $Tf= T1_X\cdot\hat{f}\circ {h^{-1}}_{|Y}$ and $T^{-1}g= T^{-1}1_Y\cdot\hat{g}\circ {h}_{|X}$.
\end{thm}

\begin{proof}
It is clear that $T$ induces a linear order isomorphism $\hat{T}:
\hat{A}_b(X) \to \hat{A}_b(Y)$ given by $\hat{T}\hat{f} =
(Tf)\hat{}$. The spaces $\hat{A}_b(X)$ and $\hat{A}_b(Y)$ are
subspaces of $C(\cA X)$ and $C(\cA Y)$ respectively. By Lemma
\ref{lem6.0}, $\hat{A}_b(X)$ and $\hat{A}_b(Y)$ precisely separate
points from closed sets. Therefore, by Theorem \ref{thm0} and the
remark following Proposition \ref{prop4}, there is a homeomorphism
$h:\cA X \to \cA Y$ such that $\hat{T}\hat{f} = (T1_X)^{\hat{}}\cdot
\hat{f}\circ h^{-1}$ for all $\hat{f}\in \hat{A}_b(X)$ and
$\hat{T}^{-1}\hat{g} = (T^{-1}1_Y)\hat{}\cdot \hat{g}\circ h$ for
all $\hat{g}\in \hat{A}_b(Y)$. Since $T^{-1}1_Y$ is bounded, we get
in particular that
\begin{equation}\label{eq1}
1_{\cA Y} = \hat{T}(T^{-1}1_Y)\hat{} = (T1_X)\hat{}\cdot (T^{-1}1_Y)\hat{}\circ h^{-1}.
\end{equation}

Let $f\in A(X)$ with $f \geq 0$. Suppose that $x_0 \in \cA X$ and
$y_0 = h(x_0) \in Y$. Note that $\hat{f}(x_0) \geq 0$. If
$\hat{f}(x_0) = 0$, then $Tf(y_0) \geq (T1_X)(y_0)\hat{f}(x_0)$.  Otherwise,
for all $a\in \R$ with $0 \leq a < \hat{f}(x_0)$, there exists an
open neighborhood $U$ of $x_0$ in $\cA X$ such that $f(x)
> a$ for all $x\in U \cap X$. By Lemma \ref{lem6.0}, there
exists $g\in A_b(X)$ such that $\rg \hat{g} \subseteq [0,1]$,
$\hat{g}(x_0) = 1$ and $\hat{g}=0$ outside $U$. Note that $(Tg)(y_0)
= (T1_X)(y_0)\hat{g}(x_0)= (T1_X)(y_0)$ since $g \in A_b(X)$ and
$y_0\in Y$. As $f -ag \geq 0$, $T(f-ag) \geq 0$. Thus $(Tf)(y_0)
\geq a(Tg)(y_0) = a(T1_X)(y_0)$. This shows that  $(Tf)(y_0) \geq
(T1_X)(y_0)\hat{f}(x_0)$. Since $T1_X(y_0) = (T1_X)\hat{}(y_0) > 0$
by Proposition \ref{prop3.5}, we see that, in particular,
$\hat{f}(x_0) \in \R$. In other words, $\hat{f}(h^{-1}(y)) \in \R$
and $(Tf)(y) \geq (T1_X)(y)\hat{f}(h^{-1}(y))$ for all $y\in Y$. By
symmetry and equation (\ref{eq1}), we also get that for all $x\in
X$,
\begin{align*}
(T1_X)\hat{}(h(x))\cdot {f}(x) &= (T1_X)\hat{}(h(x))\cdot (T^{-1}(Tf))(x) \\
&\geq (T1_X)\hat{}(h(x))\cdot (T^{-1}1_Y)(x)\cdot (Tf)\hat{}(h(x))\\
&= (Tf)\hat{}(h(x)).
\end{align*}
Given $y\in Y$, let $(x_\al)$ be a net in $X$ so that $(h(x_\al))$ converges to $y$.  Applying the preceding calculation to $x_\al$ and taking limit gives $T1_X(y)\cdot \hat{f}(h^{-1}(y)) \geq (Tf)(y)$.
Thus  $(Tf)(y)  =   (T1_X)(y)\hat{f}(h^{-1}(y))$ for all $y\in Y$.

For a general $f\in A(X)$, write $f = f_1-f_2$, where $0 \leq f_1, f_2 \in A(X)$. If $y\in Y$, $\hat{f_1}(h^{-1}(y))$ and $\hat{f_2}(h^{-1}(y)) \in \R$.  By the previous paragraph,
\begin{align*} Tf(y) &= Tf_1(y) - Tf_2(y) = T1_X(y)\cdot(\hat{f_1}(h^{-1}(y)) - \hat{f_2}(h^{-1}(y)))\\ &= T1_X(y)\cdot\hat{f}(h^{-1}(y)).
\end{align*}

The formula for $T^{-1}g$ follows by the same argument.
\end{proof}

\noindent{\bf Remark}. If $T:A(X)\to A(Y)$ is a linear order isomorphism so that there exists $0 < c < 1$ so that $c1_Y \leq T1_X\leq c^{-1}1_Y$, then $T(A_b(X)) = A_b(Y)$.  This holds in particular if $T1_X = 1_Y$.
\begin{proof}
In fact, if $f \in A_b(X)$, then there exists $0 < M < \infty$ such that $-M1_X \leq f  \leq M1_X$.  Then
\[ -Mc^{-1}1_Y \leq -MT1_X \leq Tf \leq MT1_X \leq Mc^{-1}1_Y.\]
Since the condition $c1_Y \leq T1_X\leq c^{-1}1_Y$ is equivalent to
$c1_X \leq T^{-1}1_Y \leq c^{-1}1_X$, the other direction follows by
symmetry.
\end{proof}

Theorem \ref{thm6} applies to all function spaces that are commonly
considered in the context of order isomorphisms.  Given a function
space $A(X)$, let $A^{\loc}(X)$ be the space of all real-valued
functions $f$ on $X$ such that for every $x_0 \in X$, there are a
neighborhood $U$ of $x_0$ and a function $g \in A(X)$ such that $f =
g$ on the set $U$. The space $A^{\loc}_b(X)$ is the subspace of all
bounded functions in $A^{\loc}(X)$.

\begin{prop}\label{prop6.1}
\begin{enumerate}
\item Let $X$ be a completely regular space and let $A(X)$ be a sublattice of $C(X)$ that separates points from closed sets and contains the constant functions.  Then $A(X)$ is adequate.
\item If $A(X)$ is adequate, then so are $A_b(X)$, $A^{\loc}(X)$ and $A^{\loc}_b(X)$.
\item Let $X$ be an open set in a Banach space $E$ and let $p \in \N \cup \{\infty\}$. Suppose that $E$ supports a $C^p$ bump function, i.e., there exists $\vp \in C^p(E)$ such that $\vp(0)\neq 0$ and that $\vp$ has bounded support.  Then $C^p(X)$ is adequate.
\end{enumerate}
\end{prop}

\begin{proof}
For part (a), see the remark following Definition
\ref{def:adequate}. Part (b) is clear.  For part (c), take $g:\R\to
\R$ to be a nondecreasing $C^\infty$ function such that $g(t) = 0$
if $t \leq 0$ and $g(t) = 1$ if $t \geq 1$. Then $C^p(X)$ and
$C^p_b(X)$ are $g$-invariant.
\end{proof}

Let $(X,d)$ be a metric space.  The space of Lipschitz functions on
$X$, $\Lip(X)$, consists of all real-valued functions $f$ on $X$
such that
\[\sup\bigl\{\frac{|f(x)-f(y)|}{d(x,y)}: x, y \in X, x\neq y\bigr\} < \infty.\]
The space $\lip(X)$ of little Lipschitz functions consists of all
$f\in \Lip(X)$ such that
\[ \lim_{t\to 0}\sup\{\frac{|f(x)-f(y)|}{t}: 0 < d(x,y) < t\} = 0.\]
The space of uniformly continuous functions on $X$ is denoted by
$UC(X)$. $\Lip(X)$, $\lip(X)$ and $UC(X)$ are sublattices of $C(X)$
that contain the constant functions. $\Lip(X)$ and $UC(X)$ always
separate points from closed sets and hence are adequate by
Proposition \ref{prop6.1}. If $\lip(X)$ separates points from closed
sets, then it is also adequate.  This occurs in particular if there
exists $0 < \al <1$ and a metric $D$ on $X$ so that $d = D^\al$. 
By Proposition \ref{prop6.1}(c), the
local, bounded, and bounded local versions of these spaces are also
adequate.

An obvious question with regard to Theorem \ref{thm6} is under what
circumstances  would the homeomorphism $h$ map $X$ onto $Y$.  A
simple example shows that this may not always be the case.  In fact,
for any completely regular space $X$, $C_b(X)$ is linearly order
isomorphic to $C(\beta X)$.  Thus, if $X$ and $Y$ are
non-homeomorphic spaces with homeomorphic Stone-\v{C}ech
compactifications, then $C_b(X)$ is linearly order isomorphic to
$C_b(Y)$ under a linear isomorphism that preserves bounded
functions.

\section{Refinements}\label{sec3}

The purpose of this section is to refine Theorem \ref{thm6} by showing that in many situations, the homeomorphism $h$ maps $X$ onto $Y$.  We will also show that in some cases, one may remove the condition that $T$ preserves bounded functions. First we look at a classical situation.
A completely regular space $X$ is said to be {\em realcompact} if given $x \in \beta X \bs X$, there exists $f \in C(X)$ whose extension $\hat{f}$ takes infinite value at $x$.

\begin{thm}\label{thm7}
Let $X$ and $Y$ be realcompact spaces and let $T: C(X) \to C(Y)$ be a linear order isomorphism.  Then there exists a homeomorphism $\theta: X\to Y$ such that $Tf = T1_X\cdot f\circ \theta^{-1}$ and $T^{-1}g = T^{-1}1_Y\cdot g\circ \theta$ for all $f\in C(X)$ and all $g\in C(Y)$.
\end{thm}

\begin{proof}
Let $u = 1_X+ T^{-1}1_Y$.  Observe that $u \geq 1_X$ and $Tu \geq 1_Y$. Define $S: C(X) \to C(Y)$ by $Sf = T(uf)/Tu$ for all $f\in C(X)$.  It is easy to check that $S$ is a linear order isomorphism such that $S1_X = 1_Y$.   Let  $h:\beta X \to \beta Y$ be the homeomorphism obtained by applying Theorem \ref{thm6} to the map $S$.
Since $\hat{f}(h^{-1}(y))\R$ for all $f\in C(X)$ and $y \in Y$, and $X$ is realcompact, we find that $h^{-1}(y) \in X$ for all $y\in Y$.  Similarly, $h(x) \in Y$ for all $x\in X$.  Let $\theta$ be the restriction of $h$ to $X$.  Then $\theta$ is a homeomorphism from $X$ onto $Y$.
By Theorem \ref{thm6} again,
and $Sf = f\circ \theta^{-1}$for all $f\in C(X)$.
In particular, $(1/u)\circ \theta^{-1} = S(\frac{1}{u}) = \frac{T1_X}{Tu}$.
Hence, for all $f\in C(X)$,
\begin{align*}
Tf &= Tu \cdot S(\frac{f}{u}) = Tu \cdot \frac{f}{u}\circ \theta^{-1} = Tu \cdot(f\circ \theta^{-1})\cdot (\frac{1}{u}\circ \theta^{-1})\\& = Tu \cdot(f\circ \theta^{-1})\cdot \frac{T1_X}{Tu} = T1_X\cdot (f\circ \theta^{-1}).
\end{align*}
Similarly, $T^{-1}g = T^{-1}1_Y \cdot g\circ \theta$ for all $g\in C(Y)$.
\end{proof}

Representation of nonlinear order isomorphisms between spaces of
continuous functions on compact Hausdorff spaces has been obtained
by F.\ Cabello S\'{a}nchez \cite{CS1}.

For the remainder of this section, we consider metric spaces $X$ and $Y$.

\begin{prop}\label{prop8.0}
Let $X$ and $Y$ be metric spaces and let $A(X)$ and $A(Y)$ be adequate subspaces of $C(X)$ and $C(Y)$ respectively.  If $h:\cA X \to \cA Y$ is a homeomorphism and $x_0\in X$, then there is sequence $(y_n)$ in $Y$ that converges to $h(x_0)$ in $\cA Y$.
\end{prop}

\begin{proof}
In the first instance, suppose that $x_0$ is an isolated point in $X$.  Since $A(X)$ separates points from closed sets, the characteristic function $\chi_{\{x_0\}} \in A(X)$.  Let $U = \{x\in \cA X: x(\chi_{\{x_0\}}) > 0\}$.  Then $U$ is an open neighborhood of $x_0$ in $\cA X$.    We claim that $U$ contains only $x_0$.  To this end, suppose that $x\in \cA X \bs\{x_0\}$.  Choose a net $(x_\al)$ in $X$ that converges to $x$.  Since $\chi_{\{x_0\}}(x_\al) \in \{0,1\}$ and $\lim_\al \chi_{\{x_0\}}(x_\al) = x(\chi_{\{x_0\}})$, the latter value is either $0$ or $1$.  If it is $1$, then there exists $\al_0$ such that $ \chi_{\{x_0\}}(x_\al) =1$ for all $\al \geq \al_0$.  Then $x_\al = x_0$ for all $\al \geq \al_0$.  As a result, $x= x_0$, yielding a contradiction.  Thus $x(\chi_{\{x_0\}}) = 0$, i.e., $x\notin U$, as claimed.
By the claim, $x_0$ is an isolated point in $\cA X$.  Hence $h(x_0)$ is an isolated point in $\cA Y$.  But since $Y$ is dense in $\cA Y$, $h(x_0)$ cannot be in $\cA Y \bs Y$.  So $h(x_0) \in Y$ and the conclusion of the proposition is obvious.

Now suppose that $x_0$ is not an isolated point in $X$.  Fix a pairwise distinct sequence of points $(x_n)$ in $X$ that converges to $x_0$ and a strictly positive null sequence $(\ep_n)$ in $\R$. Choose $f_n \in A(X)$ such that $f_n(x_n) = 1$ and $f_n = 0$ outside $B(x_n,\ep_n)$.  Set $U_n = \{x\in \cA X: x(f_n)> 0\}$.  Then $U_n$ is an open neighborhood of $x_n$ in $\cA X$.  Thus $h(U_n)$ is a nonempty open set in $\cA Y$ and hence $h(U_n) \cap Y \neq \emptyset$.
Pick $y_n \in h(U_n)\cap Y$ and let $z_n = h^{-1}(y_n)$.
Take any $f\in A(X)$.  In particular, $f$ is continuous on $X$.  For any $\ep > 0$, there exists $\delta > 0$ so that $|f(x) - f(x_0)| < \ep$ if $x\in B(x_0,\delta)$.
Observe that $U_n$ is open in $\cA X$ and hence $U_n \subseteq \ol{U_n\cap X}$.  Then $z_n \in U_n \subseteq \ol{U_n\cap X}$.  By continuity of $\hat{f}$, $z_n(f) = \hat{f}(z_n) \in \ol{f(U_n\cap X)}$.
There exists $n_0$ such that $B(x_n,\ep_n) \subseteq B(x_0,\delta)$ for all $n\geq n_0$.
Hence $U_n\cap X  \subseteq B(x_0,\delta)$ for all $n\geq n_0$.
Therefore, $|z_n(f)-f(x_0)| \leq  \ep$ for all $n \geq n_0$.
This proves that $\lim z_n(f) = f(x_0)$.
As $f\in A(X)$ is arbitrary, $z_n \to x_0$ in $\cA X$.
Thus $y_n = h(z_n) \to h(x_0)$ in $\cA Y$, as desired.
\end{proof}

A set of points $S$ in a metric space is {\em separated} if there exists $\ep > 0$ such that $d(x_1,x_2) \geq \ep$ whenever $x_1$ and $x_2$ are distinct points in $S$.

\begin{cor}\label{cor9}
Let $X$  and $Y$ be metric spaces and let $A(X)$ and $A(Y)$ be
adequate subspaces of $C(X)$ and $C(Y)$ respectively.  Assume that
\begin{enumerate}
\item $A(Y) = A^{\loc}(Y)$ or $A^{loc}_b(Y)$, or
\item $Y$ is complete and that for any separated sequence $(y_n)$ in $Y$, there exists $g\in A(Y)$ such that $g(y_{2n}) = 1$ and $g(y_{2n-1}) = 0$ for all $n$.
\end{enumerate}
If $h:\cA X \to \cA Y$ is a homeomorphism and $x_0\in X$, then $h(x_0) \in Y$.
\end{cor}

\begin{proof}
(a) By Proposition \ref{prop8.0}, there is a sequence $(y_n)$ in $Y$
that converges to $h(x_0)$. If $(y_n)$ has a subsequence that
converges in $Y$, then we are done. Assume that $(y_n)$ has no
subsequence that converges in $Y$. We may then assume that there is
a strictly positive real sequence $(\ep_n)$ so that $B(y_m,2\ep_m)
\cap B(y_n,2\ep_n) = \emptyset$ if $m\neq n$. By Lemma \ref{lem6.0},
there exists $g_n \in A(Y)$ so that $g_n(y_n )=1$, $g_n = 0$ outside
$B(y_n,\ep_n)$, and $0 \leq g_n \leq 1_Y$. Let $g$ be the pointwise
sum $\sum g_{2n}$. Take $y\in Y$.  If $y \in B(y_n,2\ep_n)$ for some
$n$, then there exists $\ep > 0$ such that $B(y,\ep)$ does not
intersect $B(y_m,\ep_m)$ for any $m \neq n$. Suppose $y \notin
B(y_n,2\ep_n)$ for any $n$.  If, for all $\ep > 0$, $B(y,\ep)$
intersects $B(y_n,\ep_n)$ for at least two $n$, then for any $\ep >
0$, $B(y,\ep)$ intersects infinitely many $B(y_n,\ep_n)$.  This
implies that $(y_n)$ has a convergent subsequence, contrary to the
choice of $(y_n)$.  This establishes that
 there exists $\ep
> 0$ so that $B(y,\ep)$ intersects at most one $B(y_n,\ep_n)$.
Thus $g$ is a bounded function in $A^{\loc}(Y)$. By the assumption,
$g\in A(Y)$.  In particular, $(g(y_n))$ converges to
$\hat{g}(h(x_0))$.  However, this is impossible since $g(y_n) = 1$
if $n$ is even and $0$ if $n$ is odd.

\noindent (b) By Proposition \ref{prop8.0}, there is a sequence
$(y_n)$ in $Y$ that converges to $h(x_0)$. If $(y_n)$ has a Cauchy
subsequence, then we are done. Otherwise, we may assume that $(y_n)$
is separated. By the assumption, there exists $g\in A(Y)$ so that
$g(y_{2n}) = 1$ and $g(y_{2n-1}) = 0$ for all $n$. But this is
impossible since  $(g(y_n))$ converges to $\hat{g}(h(x_0))$.
\end{proof}

The next theorem unifies many results concerning unital order isomorphisms on most types of commonly considered function spaces defined on metric spaces.

\begin{thm}\label{thm16.0}
Let $X$ and $Y$ be metric spaces.  Assume that $A(X)$ is an adequate subspace of $C(X)$ and that either
\begin{enumerate}
\item $A(X) = A^{\loc}(X)$ or $A^{\loc}_b(X)$, or
\item $X$ is complete and  for any separated sequence $(x_n)$ in $X$, there exists $f\in A(X)$ such that $f(x_{2n}) = 1$ and $f(x_{2n-1}) = 0$ for all $n$.
\end{enumerate}
Assume the same for $A(Y)$.
If $T:A(X)\to A(Y)$ is a linear order isomorphism such that $T(A_b(X)) = A_b(Y)$, then
there is a homeomorphism $h:X \to Y$ such that $Tf = T1_X\cdot f\circ h^{-1}$ and $T^{-1}g = T1_Y\cdot g\circ h$ for all $ f\in A(X)$ and $g\in A(Y)$.
\end{thm}

\begin{proof}
Apply Theorem \ref{thm6} and Corollary \ref{cor9} (a) or (b) as the case may be.
\end{proof}

Theorem \ref{thm16.0} applies if $A(X)$ (and $A(Y)$) is of one of the following types:
\begin{enumerate}
\item $\Lip(X), \Lip_b(X)$ $UC(X)$, $UC_b(X)$, where $X$ is complete metric;
\item $\lip(X), \lip_b(X)$, where $X$ is complete metric and $\lip(X)$ satisfies condition (b) in Theorem \ref{thm16.0};
\item $\Lip^{\loc}(X)$, $\Lip^{\loc}_b(X)$, $UC^{\loc}(X)$, $UC^{\loc}_b(X)$, where $X$ is metric (not necessarily complete);
\item $\lip^{\loc}(X)$, $\lip^{\loc}_b(X)$, where $X$ is metric, not necessarily complete, and $\lip^{\loc}(X)$ separates points from closed sets;
\item $C^p(X)$, $C^p_b(X)$, where $X$ is an open set in a Banach space $E$ that supports a $C^p$ bump function.
\end{enumerate}

Finally, we show that the condition that $T$ preserves bounded functions may be removed in certain cases.  The idea is to use the ``division trick" that has been employed in the proof of Theorem \ref{thm7}.

\begin{thm}\label{thm17}
Let $X$ and $Y$ be metric spaces.
Assume that
\begin{enumerate}
\item $A(X)$ is a sublattice of $C(X)$ that separates points from closed sets and contains constants, and that either
\begin{enumerate}
\item $A(X) = A^{\loc}(X)$ or $A^{\loc}_b(X)$, or
\item $X$ is complete and $A(X) = \Lip(X)$ or $\Lip_b(X)$; or
\item $A(X) = \lip(X)$ or $\lip_b(X)$, where $X$ is complete with metric $d$ such that $d = D^\al$ for some metric $D$ on $X$ and $0 < \al < 1$; or
\end{enumerate}
\item $A(X) = C^p(X)$ or $C^p_b(X)$, where $X$ is an open set in a Banach space $E$ that supports a $C^p$ bump function.
\end{enumerate}
Assume the same for $A(Y)$.
If $T:A(X)\to A(Y)$ is a linear order isomorphism, then
there is a homeomorphism $h:X \to Y$ such that $Tf = T1_X\cdot f\circ h^{-1}$ and $T^{-1}g = T1_Y\cdot g\circ h$ for all $ f\in A(X)$ and $g\in A(Y)$.
\end{thm}

\begin{proof}
Let $u = 1_X+ T^{-1}1_Y$.  Then $u\in A(X)$ and $u \geq 1_X, Tu \geq 1_Y$. Define
\[
F(X) = \bigl\{\frac{f}{u}: f \in A(X)\bigr\} \quad\text{and}\quad
F(Y) = \bigl\{\frac{g}{Tu}: g\in A(Y)\bigr\}.\] If $A(X)$ is a
sublattice of $C(X)$ that separates points from closed sets, then
$F(X)$ is a sublattice of $C(X)$ that contains constants and
separates points from closed sets.  If $A(X) = C^p(X)$, then $F(X) =
C^p(X)$.  If $A(X) = C^p_b(X)$, then $u$ is bounded and $\geq 1_X$.
Hence $F(X) = C^p_b(X)$.  Thus, in all cases, $F(X)$ is an adequate
subspace of $C(X)$.  The same considerations apply to $F(Y)$. Define
a map $S: F(X) \to F(Y)$  by $Sf = T(uf)/Tu$ for all $f\in F(X)$. It
is easy to check that $S$ is a linear isomorphism such that $S1_X =
1_Y$.  From the remark following Theorem \ref{thm6}, we see that
$S({F}_b(X)) = {F}_b(Y)$. Denote the $F(X)$- and $F(Y)$-
compactifications of $X$ and $Y$ respectively by $\cF X$ and $\cF Y$
respectively. By Theorem \ref{thm6}, there is a homeomorphism $h:
\cF X \to \cF Y$ such that $Sf= \hat{f}\circ {h^{-1}}_{|Y}$ for all
$f \in F(X)$. Thus
\begin{equation}\label{eq2}
Tf(y) = Tu(y)\cdot (f/u)\hat{}(h^{-1}(y))\quad \text{ for all $f\in A(X)$ and all $y \in Y$}.
\end{equation}
If $h$ is a homeomorphism from $X$ onto $Y$, then it follows from
(\ref{eq2}) that $T1_X(y) = Tu(y)/u(h^{-1}(y))$ and hence $Tf(y) =
T1_X(y)\cdot f(h^{-1}(y))$. The result for $T^{-1}g$ can be obtained
similarly.  By symmetry, it remains to show that $h(X) \subseteq Y$.

If $A(Y) = A^{\loc}(Y)$ or $A^{\loc}_b(Y)$ (including  the cases
where $A(Y) = C^p(Y)$ or $C^p_b(Y)$), then $F(Y) = F^{\loc}(Y)$ or
$F^{\loc}_b(Y)$ (note that in the latter case $Tu$ is bounded and
also bounded away from $0$). So by Corollary \ref{cor9}, $h(X)
\subseteq Y$.

Finally, we consider case (a)(ii).  For case (a)(iii), a similar
argument works using the metric $D$. Suppose that $Y$ is complete
and $A(Y) = \Lip(Y)$ or $\Lip_b(Y)$.  Let $x_0 \in X$ and $y_0 =
h(x_0)$. By Proposition \ref{prop8.0}, there exists a sequence
$(y_n)$ in $Y$ that converges to $y_0$ in $\cF Y$. If $(y_n)$ has a
Cauchy subsequence in $Y$, then $y_0 \in Y$.  Otherwise, by using a
subsequence, we may assume that $(y_n)$ is separated. In the first
instance, suppose that $(Tu(y_n))$ has a bounded subsequence.  Then
we may assume without loss of generality that $(Tu(y_n))$ converges
to a real number $a$. Applying (\ref{eq2}) with $y = y_n$ and taking
limit, we see that $\lim Tf(y_n) = a \cdot f(x_0)/u(x_0) \in \R$ for
any $Tf \in A(Y)$. But since $(y_n)$ is a separated sequence in $Y$
and $\Lip_b(Y) \subseteq A(Y)$, this is not true. Hence it must be
that $(Tu(y_n))$ diverges to $\infty$. By choosing a subsequence if
necessary, we may assume that $Tu(y_n) > 4Tu(y_m)$ if $m < n$. Note
that there is a constant $C >0$ such that $d(y_m,y_n) \geq
C|Tu(y_m)-Tu(y_n)|$.  Then
\[ d(y_m,y_n) \geq \frac{C}{2}(Tu(y_m)+Tu(y_n))\quad\text{if $m < n$}.\]
Hence the balls $B(y_n, \frac{C}{2}Tu(y_n))$ are pairwise disjoint.
Furthermore, as $Tu$ is unbounded, $A(Y) = \Lip(Y)$.  In this case,
it is easy to construct a function $g$ in $\Lip(Y)$ such that
$g(y_{2n-1}) = \frac{C}{2}Tu(y_{2n-1})$ and $g(y_{2n}) = 0$.
However, it follows from (\ref{eq2}) that $\lim Tf(y_n)/Tu(y_n) =
f(x_0)/u(x_0)$ for any $f \in A(X)$.  Taking $f = T^{-1}g$ yields a
contradiction.
\end{proof}

\end{document}